\documentclass[leqno,twoside]{amsart}
\newif\ifdraft
\draftfalse

\usepackage{fixltx2e,amssymb,indentfirst,xspace,bm}
\usepackage[retainorgcmds]{IEEEtrantools}
\usepackage{framed,paralist} 
\usepackage{mathtools}
\ifdraft
	\usepackage[notref,notcite]{showkeys}
	\usepackage{datetime,currfile}
\fi
\usepackage[normalem]{ulem}
\usepackage[pagebackref=true, colorlinks=true, citecolor=blue, pdfborder={0 0 .1},breaklinks]{hyperref}

\usepackage[initials,nobysame]{amsrefs}

\makeatletter
\def\@wraptoccontribs#1#2{}

\@mparswitchfalse
\makeatother

\ifdraft
    \newcommand{\sidenote}[1]{\marginpar[\raggedleft\tiny #1]{\raggedright\tiny #1}}

\else
    \newcommand{\sidenote}[1]{}
    
\fi
\newcommand{\warning}[1]{\typeout{}\typeout{WARNING: #1 at line \the\inputlineno}\typeout{}}
\newenvironment{note}[1][TODO. ]{%
    \ifdraft\else\warning{Note environment still present in final version}\fi
    \MakeFramed{\advance\hsize-\width \FrameRestore}\noindent\textbf{#1}}%
    {\endMakeFramed}

\makeatletter
\newcommand{\UWave}[2][blue]{\bgroup \markoverwith{\textcolor{#1}{\lower3.5\p@\hbox{\sixly \char58}}}\ULon{#2}}
\makeatother
\newcommand{\SOut}[2][red]{\bgroup\markoverwith {\textcolor{#1}{\rule[.45ex]{2pt}{.1ex}}}\ULon{#2}}

\newcommand{\highlight}[2][yellow]{\bgroup\markoverwith {\textcolor{#1}{\rule[-.2em]{2pt}{1.2em}}}\ULon{#2}}

\newenvironment{beqn}[1][:C?s]%
    {\left\{\begin{IEEEeqnarraybox}[\relax][c]{#1}}%
    {\end{IEEEeqnarraybox}\right.}

\newcommand{\del}{\partial}
\newcommand{\delt}{\prt_t}

\newcommand{\lap}{\Delta}

\newcommand{\grad}{\nabla}

\newcommand{\divergence}{\dv}

\newcommand{\cross}{\times}

\renewcommand{\epsilon}{\varepsilon}
\renewcommand{\leq}{\leqslant}
\renewcommand{\geq}{\geqslant}

\newcommand{\R}{\mathbb{R}}

\newcommand{\N}{\mathbb{N}}

\newif\iftextstyle
\textstyletrue
\everydisplay\expandafter{\the\everydisplay\textstylefalse}

\DeclarePairedDelimiter{\abs}{\lvert}{\rvert}
\DeclarePairedDelimiter{\norm}{\lVert}{\rVert}
\DeclarePairedDelimiter{\average}{\langle}{\rangle}
\newcommand{\dualpairing}[2]{\average{#1,#2}}
\newcommand{\ip}[2]{(#1,#2)}

\newcommand{\defeq}{\stackrel{\scriptscriptstyle \text{def}}{=}}

\numberwithin{equation}{section}
\allowdisplaybreaks

\newtheorem{theorem}{Theorem}[section]

\newtheorem{lemma}[theorem]{Lemma}
\newtheorem{proposition}[theorem]{Proposition}
\newtheorem{corollary}[theorem]{Corollary}

\newtheorem*{theorem*}{Theorem}
\newtheorem*{lemma*}{Lemma}
\newtheorem*{proposition*}{Proposition}
\newtheorem*{corollary*}{Corollary}

\theoremstyle{definition}
\newtheorem{definition}[theorem]{Definition}

\theoremstyle{remark}
\newtheorem{remark}[theorem]{Remark}
\newtheorem*{remark*}{Remark}

\newtheoremstyle{cases}
  {\smallskipamount}
  {\smallskipamount}
  {}
  {\parindent}
  {\itshape}
  {.}
  {.5em}
  {\thmnote{#1 #2: #3}}
\newcounter{CaseCounter}
\theoremstyle{cases}
\newtheorem{case}[CaseCounter]{Case}

\newcommand{\eqsysSR}{\eqref{eqnSRu}--\eqref{eqnSRC}\xspace}
\newcommand{\eqsysWeakS}{\eqref{eqnWeakSu}--\eqref{eqnWeakSdivergence}\xspace}

\newcommand{\Ignore}[1]{}
\newcommand{\iny}{\ensuremath{\infty}}
\newcommand{\dv}{\grad\cdot} %
\newcommand{\prt}{\ensuremath{\partial}}
\newcommand{\brac}[1]{\ensuremath{\left[ #1 \right]}}
\newcommand{\pr}[1]{\ensuremath{\left( #1 \right) }}
\newcommand{\set}[1]{\ensuremath{\left\{ #1 \right\}}}

\newcommand{\refS}[1]{Section~\ref{sxn#1}}
\newcommand{\refT}[1]{Theorem~\ref{thm#1}}
\newcommand{\refL}[1]{Lemma~\ref{lma#1}}

\newcommand{\refD}[1]{Definition~\ref{dfn#1}}

\newcommand{\refE}[1]{\eqref{eqn#1}}
\newcommand{\refEAnd}[2]{\eqref{eqn#1} and \eqref{eqn#2}}
\newcommand{\refR}[1]{Remark~\ref{rmk#1}}
\newcommand{\Cal}[1]{\ensuremath{\mathcal{#1}}}

\newcommand{\la}{\ensuremath{\lambda}}
\newcommand{\ol}{\overline}

\newcommand{\LP}{\ensuremath{\Cal{P}}} %
\newcommand{\lhp}{\LP} %
\newcommand{\n}{\bm{n}}
\newcommand{\LTwon}{L^2_{\n}}

\newcommand{\dt}{\ensuremath{\frac{d}{dt}}}

\newcommand{\smallnorm}[1]{\ensuremath{\Vert #1 \Vert}}

\newcommand{\diff}[2]{\frac{ d#1}{d#2}}

\newcommand{\innp}[1]{\ensuremath{\left< #1 \right>}}

\newcommand{\extended}{extended\xspace}

\begin{document}
\ifdraft
	\title
		[\currfilename \qquad \mmddyyyydate\today \quad\currenttime]
		{Global existence for two extended Navier-Stokes systems}
\else
	\title
		[Global existence for extended Navier-Stokes systems]
		{Global existence for two extended Navier-Stokes systems}
\fi

\author[Ignatova]{Mihaela Ignatova}
\address{Department of Mathematics, Stanford University, Stanford, CA 94305}
\email{mihaelai@stanford.edu}

\author[Iyer]{Gautam Iyer}
\address{Department of Mathematical Sciences, Carnegie Mellon University, Pittsburgh PA 15213}
\email{gautam@math.cmu.edu}

\author[Kelliher]{James P. Kelliher}
\address{Department of Mathematics, University of California, Riverside, 900 University Ave., Riverside, CA 92521}
\email{kelliher@math.ucr.edu}

\author[Pego]{Robert L. Pego}
\address{Department of Mathematical Sciences, Carnegie Mellon University, Pittsburgh PA 15213}
\email{rpego@math.cmu.edu}

\author[Zarnescu]{Arghir D. Zarnescu}
\address{Department of Mathematics, University of Sussex, Falmer, Brighton, BN1 9QH,
United Kingdom}
\email{A.Zarnescu@sussex.ac.uk}

\thanks{This material is based upon work partially supported by the National Science Foundation under grants
DMS-0905723,	 
DMS-1007914, 
DMS-1158938, 
DMS-1212141, 
DMS-1211161, 
DMS-1252912. 
GI also acknowledges partial support from an Alfred P. Sloan research fellowship.
The authors also thank the Center for Nonlinear Analysis (NSF Grants No. DMS-0405343 and DMS-0635983), where part of this research was carried out.}

\begin{abstract}
	We prove global existence of weak solutions to two systems of equations which extend the dynamics of the Navier-Stokes equations for incompressible viscous flow with no-slip boundary condition.
	The systems of equations we consider arise as formal limits of time discrete pressure-Poisson schemes introduced by Johnston \& Liu 
(J.\ Comp.\ Phys.\ 199 (2004) 221--259)
and by Shirokoff \& Rosales 
(J.\ Comp.\ Phys.\ 230 (2011) 8619--8646)
	when the initial data does not satisfy the required compatibility condition.
	Unlike the results of Iyer {\it et al}\ (J.\ Math.\ Phys.\ 53 (2012) 115605),
our approach proves existence
	of \emph{weak} solutions in domains with less than $C^1$ regularity.
	Our approach also addresses uniqueness in 2D and higher regularity.
\end{abstract}


\maketitle

\ifdraft
	\begin{center}
		*** Table of contents only in draft mode, which is useful while editing ***.
	\end{center}

	\tableofcontents
	\newpage
\fi

%
%
\section{Introduction}

The pressure in incompressible fluids is a recurring source of difficulty in
theory and
numerics alike.
Formally, the pressure is a Lagrange multiplier that preserves incompressibility, and when studying well-posedness it is usually ``projected out'' and eventually recovered from the velocity field using the incompressibility constraint (see~\cite{bblLeray}).
This paper studies two systems of equations that extend the incompressible Navier-Stokes equations by specifying an explicit equation for the pressure.
Both systems arose in a numerical context (see~\cites{bblJohnstonLiu,bblShirokoffRosales}) in order to propose a time-discrete pressure-Poisson scheme (see~\cite{bblHarlowWelch}) aiming to be both efficient and accurate in domains with boundary.

For clarity of presentation, we focus this introduction on the system in~\cite{bblJohnstonLiu}, deferring the discussion of the system in~\cite{bblShirokoffRosales} to Section~\ref{sxnSR}.
The formal time-continuous limit of the time-discrete scheme proposed in~\cite{bblJohnstonLiu} turns out to be 
one of several ``reduced'' models studied by Grubb and Solonnikov~\cites{bblGrubbSolonnikov2,bblGrubbSolonnikov3}.
Explicitly, the system of interest is
\begin{align}
	\label{eqnENS1}
		\prt_t u + \LP[ (u \cdot \grad) u ] - \nu \lap u + \nu \grad p_s &= \LP f
			&& \text{in }\Omega,\\
	\label{eqnENS1bc}
		u &= 0 && \text{on } \del \Omega,	 \\
	\label{eqnENS1id}
		u \bigr|_{t = 0} &= u_0
			&& \text{in } \Omega,
\end{align}
where $p_s$, the \textit{Stokes pressure}, satisfies
\begin{equation}\label{eqnStokesPressure}
	\begin{beqn}
		-\lap p_s = 0 & in $\Omega$,\\
		\n \cdot \grad p_s = (\lap u - \grad \dv u) \cdot \n &on $\prt \Omega$.
	\end{beqn}
\end{equation}
Here $\nu > 0$ is the viscosity, $\LP$ is the Leray projection, and $f$ is the (given) external forcing.
As shown in~\cite{bblLiuLiuPego}, the Stokes pressure is alternatively represented by the formulae
\begin{equation}\label{eqnStokesPressure2}
\grad p_s = (I-\LP)(\lap u - \grad\dv u) = (\lap\LP-\LP\lap)u.
\end{equation}
Thus we can recast the  equation~\eqref{eqnENS1} as
\begin{gather}\label{eqnENSSimpleForm}\tag{\ref{eqnENS1}$'$}
		\prt_t u + \LP (u \cdot \grad u - f) = \nu(\LP \Delta u + \grad \dv u)
			\qquad \text{in }\Omega.
\end{gather}
Note the presence of the extra, stabilizing term, $\grad \dv u$, appearing on the right. (For more details see \cite{bblLiuLiuPego}.)

When the compatibility condition $\dv u_0 = 0$ is imposed on the initial data, the above system reduces exactly to the incompressible Navier-Stokes equations.
To see this, note that $\divergence u$ satisfies the heat equation with Neumann boundary conditions:
\begin{equation}\label{eqndvuHeat}
  \begin{beqn}
      \prt_t \dv u = \nu \Delta \dv u & in $\Omega$, \\
      \n \cdot \grad \dv u = 0 & on $\prt\Omega$, \\
      \dv u\bigr|_{t = 0} = \dv u_0 & \text{in } $\Omega$.
  \end{beqn}
\end{equation}
Here $\n$ denotes the outward-pointing unit normal on $\del \Omega$.
The evolution equation~\eqref{eqndvuHeat}\textsubscript{1} and initial condition~\eqref{eqndvuHeat}\textsubscript{3} follow by taking the divergence of~\eqref{eqnENSSimpleForm} and \eqref{eqnENS1id} respectively.
The boundary condition~\eqref{eqndvuHeat}\textsubscript{2} follows by taking the normal trace of~\eqref{eqnENSSimpleForm} on $\del \Omega$ and using~\eqref{eqnENS1bc}.

Uniqueness for~\eqref{eqndvuHeat} now implies that if $\divergence u_0 \equiv 0$ then $\divergence u \equiv 0$, showing that \eqref{eqnENSSimpleForm} exactly reduces to the standard, incompressible Navier-Stokes equations as claimed.
On the other hand, if~$\divergence u_0$ is non-zero initially then~\eqref{eqnENSSimpleForm} extends the dynamics of the incompressible Navier-Stokes equations in a manner that damps the
divergence
exponentially in time.


This paper primarily deals with global existence of solutions to~\eqref{eqnENSSimpleForm}, \eqref{eqnENS1bc}, \eqref{eqnENS1id}.
Local well-po\-sed\-ness for strong solutions was proved by Grubb and Solonnikov~\cites{bblGrubbSolonnikov2,bblGrubbSolonnikov3} based on a theory of pseudo-differential initial-boundary value problems, and more recently in~\cite{bblLiuLiuPego} using a novel commutator estimate.
The techniques used in~\cite{bblLiuLiuPego} were extended in~\cite{bblIyerPegoZarnescu} to also prove some (conditional) global existence results.
The results available so far, however, all assumed a regular ($C^3$) domain and only concerned strong solutions.

The difficulty in proving the existence of weak solutions to~\eqref{eqnENSSimpleForm} arises, surprisingly, from the linear terms.
To elaborate on this, observe that when we multiply~\eqref{eqnENSSimpleForm} by $u$ and integrate, the linear terms give
\begin{multline}\label{eqnCoercivityFail}
	-\int_{\Omega} u \cdot \bigl( \lhp \lap u + \grad \divergence u \bigr)
		= -\int_{\Omega} \lhp u \cdot \lap \lhp u + \norm{\divergence u}_{L^2}^2\\
		= \norm{\grad \lhp u}_{L^2}^2 + \norm{\dv u}_{L^2}^2 + \int_{\del \Omega} (\lhp u) \cdot \frac{\del \lhp u}{\del \n}.
\end{multline}
For the standard incompressible Navier-Stokes equations, the no-slip boundary condition and the incompressibility constraint together guarantee $\lhp u = 0$ on $\del \Omega$, ensuring the boundary integral on the right of~\eqref{eqnCoercivityFail} vanishes.
In our situation, however, the boundary integral above provides a
high-order contribution without consistent sign.
Thus, while
$-(\LP \lap + \grad \divergence)$
is a positive coercive operator on the space of \emph{divergence-free} $H^1_0$ functions, it is neither positive nor coercive with respect to the standard inner product on $H^1_0$ (see~\cite{bblIyerPegoZarnescu}*{Proposition 2.1}).

The key idea in~\cite{bblIyerPegoZarnescu} is to use a commutator estimate to construct a nonstandard $H^1$-equivalent inner product that makes $-(\LP \lap + \grad \divergence)$ coercive on the whole space $H^1_0$, and use this to prove conditional global existence results.
Unfortunately, under these inner products, the nonlinearity is harder to control; consequently, the results in~\cite{bblIyerPegoZarnescu} are unable to effectively exploit the depletion of the nonlinearity that is available under the standard $L^2$ inner product.
For the standard Navier-Stokes equations the $L^2$ depletion of the nonlinearity is responsible for the energy inequality, which is central to almost every global existence result available.

This paper uses a different approach to study well-posedness of~\eqref{eqnENSSimpleForm}.
The central idea is to consider the $H^1_0$ (not $L^2$!) orthogonal decomposition of the solution $u$ into a divergence free part and a remainder.
Namely, we write 
\[
u = v + z, \qquad\mbox{where}\quad \divergence{v}=0, \quad  \int_{\Omega} \grad v \cdot \grad z = 0,
\]
and $v, z \in H^1_0$.
The key observation in our proof is that $z$ is completely determined by $\divergence u$, a quantity that is globally determined only from $\divergence u_0$ via~\eqref{eqndvuHeat} and is independent of $v$.
Now the evolution of $v$ is essentially a perturbed Navier-Stokes equation, which can be analyzed using well-established techniques.
In domains with Lipschitz boundaries, however, one has to tread cautiously.

The advantage of the method used in this paper is that it proves the existence of global weak solutions of~\eqref{eqnENSSimpleForm}, \eqref{eqnENS1bc}, \eqref{eqnENS1id} even in irregular (Lipschitz) domains.
The methods in~\cites{bblIyerPegoZarnescu,bblLiuLiuPego} prove the existence of strong solutions, and require $H^1_0$ initial data and a $C^3$-domain.
In this paper, we prove existence of weak solutions with $L^2$ initial data and \emph{either} $H^2$ initial divergence and a Lipschitz domain \emph{or} $L^2$ initial divergence and a $C^2$ domain.
The interest in studying weak solutions and lowering the regularity requirements of the domain is that numerical simulations are often performed in piecewise smooth (often polygonal) domains.
Such domains can't be handled using the techniques in~\cites{bblIyerPegoZarnescu,bblLiuLiuPego} but can be handled using our approach.

We remark, however, that our method does not appear to help in the analysis of stability and convergence of the time-discrete schemes in~\cites{bblJohnstonLiu,bblShirokoffRosales}.
In contrast, both~\cites{bblIyerPegoZarnescu} and~\cite{bblLiuLiuPego} (see also~\cite{bblLiuLiuPego09}) prove a stability result for an associated time-discrete scheme.

\subsection*{Plan of this paper.}
We begin in \refS{Decomposition} with the development of our two main tools, an orthogonal decomposition of vector fields in $H_0^1$ and a lifting from a given divergence to a vector field in $H_0^1$.
In \refS{Weak} we establish the existence and, in 2D, uniqueness of weak solutions to the \extended Navier-Stokes equations of \refE{ENSSimpleForm}, giving the higher regularity theory in \refS{Strong}.
In \refS{SR} we analyze another system of equations introduced by Johnston and Liu in \cite{bblJohnstonLiu} and studied by Shirokoff and Rosales in \cite{bblShirokoffRosales}. This system also extends the Navier-Stokes equations, though in a manner different from the extended equations in \refE{ENSSimpleForm}. We show how they can nonetheless be treated using the same key tools of \refS{Decomposition}. In the appendix we summarize the key facts we use regarding the regularity of solutions to the heat equations in Lipschitz domains.

%
%
\section{An \texorpdfstring{$H^1_0$}{H10} orthogonal decomposition and divergence lifting}\label{sxnDecomposition}

As mentioned earlier, our main tool to obtain a~priori estimates for~\eqref{eqnENSSimpleForm} is to split $u$ into a divergence-free field $v\in H^1_0$, and its $H^1_0$-orthogonal remainder.
In this section we describe this decomposition and a few well-known properties of it.

\subsection*{The \texorpdfstring{$H^1_0$}{H10}-orthogonal decomposition.}\label{sxnH10decomp}
Let $\Omega$ be a bounded, connected domain in $\R^d$, $d \geq 2$, whose boundary has at least Lipschitz regularity, and let $\n$ be the outward-pointing unit vector normal to the boundary.
Let $H$ and $V$ denote the usual functions spaces (see for instance~\cite{bblConstFoias}),
\begin{align}
	\label{eqnH}
	H &= \set{v \in L^2(\Omega) \colon \dv u = 0 \text{ and } u \cdot \n = 0 \text{ on } \prt \Omega},\\
	\label{eqnV}
	\llap{\qquad} V &= \set{v \in H_0^1(\Omega) \colon \dv u = 0}.
\end{align}

Leray's well known orthogonal projection, $\LP\colon L^2 \to H$, can be explicitly computed by solving Poisson problems.
Perhaps less familiar is the orthogonal projection of $H^1_0(\Omega)$ into $V$, which can be explicitly obtained by solving a stationary Stokes equation. For this purpose, we use the following lemma:
\begin{lemma}\label{lmaGaldi}
Let $\Omega \subset \R^d$, $d = 2, 3$, be a bounded, connected Lipschitz domain, $f \in H^{-1}(\Omega)$, $g \in L^2(\Omega)$, and $h \in H^{1/2}(\prt \Omega)$. Assume that the compatibility condition,
\begin{align*}
	\int_\Omega g = \int_{\prt \Omega} h \cdot \n
\end{align*}
holds. Then there exists a  solution $(z, q)$ in $(H^1(\Omega), L^2(\Omega))$ to
\begin{align}\label{eqnGaldiEq}
	\begin{beqn}
		-\Delta z + \grad q = f & in $\Omega$,\\
		\divergence z = g & in $\Omega$,\\
		z = h & on $\prt \Omega$.
	\end{beqn}
\end{align}
The vector field $z$ is unique and $q$ is unique up to an additive constant.
Moreover, normalizing $q$ to have mean zero,
\begin{align*}
	\norm{z}_{H^1} + \norm{q}_{L^2}
		\leq C \pr{
			\norm{f}_{H^{-1}} + \norm{g}_{L^2} + \norm{h}_{H^{1/2}(\prt \Omega)}
			}.
\end{align*}
\end{lemma}
\begin{proof}
	See,
	for instance, \cite{bblGaldi}*{Exercise IV.1.1}.
\end{proof}

Now, given $u \in H^1_0(\Omega)$, define $z$ to be the solution of the system,
\begin{equation}\label{eqnzq}
  \begin{beqn}
    -\Delta z + \grad q = 0 & in $\Omega$,\\
    \divergence z = g & in $\Omega$,\\
    z = 0 & on $\prt \Omega$,
  \end{beqn}
\end{equation}
where $g = \divergence u$, and let $v = u - z$. The existence of $z$ is assured by \refL{Galdi}, since the divergence theorem shows that $g$ satisfies the compatibility condition of that lemma.

We claim that $u = v + z$ gives the $H^1_0$-orthogonal decomposition into $V$ and $V^\perp$.
\refL{Galdi} gives $z \in H_0^1$, so if $u \in H^1_0$ then we must have both $v, z \in H^1_0$.
Since~\eqref{eqnzq} implies $\divergence v = 0$, we must have $v \in V$.
Using~$\ip{\cdot}{\cdot}$ to denote the usual $L^2$-inner product, observe that
$$
  \ip{\grad v}{\grad z} = -\ip{v}{\grad q} = 0,
$$
showing orthogonality of $v$ and $z$ in $H^1_0$ as desired.

\subsection*{Divergence lifting}\label{sxnLiftingLemma}

Equation~\eqref{eqnzq} explicitly determines the component of $u$ in $V^\perp$ from only $\divergence u$. We view this as a ``lifting'' of the scalar field, $g$, to the vector field, $z$. \refL{Galdi} gave the fundamental existence and regularity result for $z$, but we will have need of results for both higher and lower regularity of $g$. When $g$ is regular enough ($L^2$) and the boundary smooth enough ($C^2$), this is classical~\cites{bblTemam,bblAgmonDouglisNirenberg}.
Under lower regularity assumptions on $g$, a situation we encounter in the proof of our main theorem, we use a duality argument.

\begin{lemma}\label{lmaTemamLemma}
Let $\Omega \subset \R^d$, $d = 2, 3$, be a bounded, connected Lipschitz domain and $g \in L^2(\Omega)$ have mean zero.
  There exists a unique vector field $z \in H^1_0(\Omega)$ and a unique, mean-zero, scalar function $q \in L^2(\Omega)$ such that  $(z, q)$ solve~\eqref{eqnzq}.
  There exists a constant $c = c(\Omega)$ such that
    \begin{equation}\label{eqnH1boundz}
      \norm{z}_{H^{1}} + \norm{q}_{L^2} \leq c \norm{g}_{L^2}.
  \end{equation}

  If, further, $\Omega$ is $C^2$, then we have
  \begin{equation}\label{eqnweaklifting}
    \norm{z}_{L^2} \leq c \norm{g}_{\tilde H^{-1}},
  \end{equation}
  where $\tilde H^{-1}(\Omega)$ denotes the dual of $H^1(\Omega)$. (Note $\tilde H^{-1}(\Omega) \subsetneq H^{-1}(\Omega)$.)
  
  Moreover, if for some integer $m \geq 0$, $\Omega$ is a $C^{m+2}$ domain and $g \in H^{m+1}(\Omega)$, then $z \in H_0^1 \cap H^{m + 2}(\Omega)$, $q \in H^{m+1}(\Omega))$ and
  \begin{align}\label{eqnADNL2}
    \norm{z}_{H^{m + 2}} + \norm{q}_{H^{m+1}}
      \leq c \norm{g}_{H^{m+1}}.
  \end{align}
\end{lemma}
\begin{proof}
%

The existence of a solution to~\eqref{eqnzq} and the bound in \refE{H1boundz} are special cases of \refL{Galdi}.
The bound in \eqref{eqnADNL2} follows from Proposition I.2.3 of \cite{bblTemam}.

To prove~\eqref{eqnweaklifting}, let $v \in H^1_0$, $\tilde q \in L^2$ with $\int_\Omega \tilde q = 0$ solve the Stokes problem,
\begin{equation}\label{eqnvq}
    -\Delta v + \grad \tilde q = z, \quad \dv v = 0.
\end{equation}
Regularity of the Stokes operator~\cite{bblTemam}*{Proposition I.2.3}
gives
\begin{align}\label{eqnvh2}
    \norm{v}_{H^2} + \norm{\grad \widetilde{q}}_{L^2}
	\leq c \norm{z}_{L^2}.
\end{align}
Consequently,
\begin{align*}
  \norm{z}_{L^2}^2
      &= (z,z)
	= (z, -\Delta v + \nabla \tilde q)
	= (\nabla z,\nabla v) - (\dv z, \tilde q) \\
      &= (-\nabla q, v) - (g, \tilde q)
	\leq 0 + \norm{g}_{\tilde H^{-1}} \norm{\tilde q}_{H^1}
	\leq c \norm{g}_{\tilde H^{-1}} \norm{\nabla \tilde q}_{L^2}\\
      &
	\leq c \norm{g}_{\tilde H^{-1}} \norm{z}_{L^2}.
\end{align*}
Observe that our application of the Poincar\'e inequality to $\tilde q$ above is justified because $\tilde q$ has mean zero.
\end{proof}

\begin{remark}\label{rmkOnlyC2Use}
	In deriving \eqref{eqnweaklifting}, it is in only \refE{vh2} where we used the additional $C^2$-regularity assumption on $\del \Omega$.
\end{remark}

We will need to use the lifting lemma for time-depedent functions. However, the time dependence is a secondary issue, and the time regularity of $g$ is directly related to that of its lifting, $z$. To see this it suffices to note that the map that associates to $g$ the $z$ defined as solution of \eqref{eqnzq} is a linear map, and thus standard arguments will give the following Lemma.

\begin{lemma}\label{lmaTimeRegularityz} Let $\Omega \subset \R^d$, $d = 2, 3$, be a bounded $C^{m+1}$ domain, for $m\in\N$. 
	Assume  that $g$ lies in $L^p(0, T; H^m)$, $p$ in $[1, \iny]$, and that $g(t)$ has total mass zero
	for almost all $t$ in $[0, T]$.
	Let $(z, q)$ be the unique solution to \refE{zq} given by \refL{TemamLemma} for almost all times
	in $[0, T]$.
	Then	 we have
	\begin{align*}
		\norm{z}_{L^p(0, T; H^{m + 1})} + \norm{q}_{L^p(0, T; H^m)}
			\leq C \norm{g}_{L^p(0, T; H^m)}.
	\end{align*}
	If $\prt_t g$ lies in $L^p(0, T; H^{m})$,  and has total mass zero for almost all times in $[0,T]$ then 
	\begin{align*}
		\norm{\prt_t z}_{L^p(0, T; H^{m + 1})} + \norm{\prt_t q}_{L^p(0, T; H^{m})}
			\leq C \norm{\prt_t g}_{L^p(0, T; H^{m})}.
	\end{align*}	
	Moreover if $g\in C( [0,T];H^m)$ and has total mass zero for every time, then $z\in C( [0,T];H^{m+1})$ and 
	\begin{align*}
		\norm{z}_{C( [0, T]; H^{m + 1})} + \norm{q}_{L^p(0, T; H^m)}
			\leq C \norm{g}_{L^p(0, T; H^m)}.
	\end{align*}
				For $m=0$, in order to have the results above it suffices to assume that $\Omega$ is Lipschitz.		
\end{lemma}



%
%
\section{Global existence of weak solutions}\label{sxnWeak}
Identifying elements of $V$ with it's dual $V'$, one usually defines weak solutions to the standard (incompressible) Navier-Stokes equations by only using elements of $V$ as test functions.
This of course has the added advantage of completely eliminating the pressure from the equations.
For the extended dynamics, however, the velocity field $u$ is not incompressible, but only an element in $H^1_0$.
At first sight, one would expect the natural definition of weak solutions to involve testing against arbitrary $H^1_0$ functions.
Unfortunately, this poses a few problems. Suppose that $u\in H^1_0\cap H^2$ and $\varphi\in H^1_0$.
Using~$\dualpairing{\cdot}{\cdot}$ to denote the dual pairing between $V'$ and $V$, one has
$$
  \dualpairing{\LP \lap u}{\varphi} = \dualpairing{\lap u}{\LP \varphi}
$$
as before; now, however, $\LP \varphi \neq \varphi$ so $\LP \varphi$ need not vanish on the boundary. This would force us to introduce into the weak formulation an unwanted boundary integral.
Thus, it is still advantageous to use functions from $V$ as our space of test functions.
This, of course, will only recover the incompressible dynamics. For the remainder, we use the weak form of the heat equation for the divergence.

\begin{definition}\label{dfnWeakS}
	Let $u_0 \in L^2(\Omega)$ be such that $\dv u \in L^2(\Omega)$ and $u \cdot \n = 0$ on $\del \Omega$.
	We say that $u$ is a {\it weak solution} of the \extended Navier-Stokes equations~\eqref{eqnENSSimpleForm} with initial data $u_0$ if $u(0) = u_0$,
\begin{align*}
	u &\in C( [ 0, T]; V') \cap L^2(0, T; H_0^1),
	\qquad \delt u \in L^1(0, T; V'), \\
	\divergence u &\in C( [ 0, T]; L^2) \cap L^2(0, T; H^1),
\end{align*}
and
\begin{gather}
  \label{eqnWeakSu}
    \dt (u,\varphi)+( (u\cdot\grad)u,\varphi)=-\nu (\grad u,\grad\varphi)+\langle f,\varphi\rangle, \\
  \label{eqnWeakSdivergence}
    \dt (\dv u,q)=-\nu(\grad\dv u,\grad q),
\end{gather}
for every test function $\varphi\in V$, and every test function $q\in H^1$.
The time derivatives in \refEAnd{WeakSu}{WeakSdivergence} are weak distributional derivatives.
\end{definition}


\begin{remark*}
    By $\prt_t u$ in \refD{WeakS} we mean the weak time derivative of $u$ (see \cite{bblEvansBook}*{\S5.9.2}).
		Namely we require $\int_0^T \phi(t) \prt_t u(t) \, dt = -\int_0^T \phi'(t) u(t) \, dt$ for all scalar test functions $\phi \in C_c^\iny((0, T))$.
\end{remark*}


\begin{remark*}
	Equation~\eqref{eqnWeakSdivergence} says that $\dv u$ satisfies a weak formulation of the heat equation with Neumann boundary conditions.
\end{remark*}

\begin{remark*}
The weak formulation above may be compared to one  
described by Sani {\it et al.}~\cite{bblSani}
for a Stokes system with a pressure Poisson equation but with zero divergence
constraint. 
\end{remark*}

Our main theorem is the existence of weak solutions to~\eqsysWeakS.
\begin{theorem}[Global existence of weak solutions]\label{thmExistence}
Let $\Omega\subset\R^d$, $d=2,3$, be a bounded domain, and let $T > 0$ be arbitrary.  
Suppose that $u_0 \in L^2(\Omega)$ and $u_0$, $f$ satisfy
  \begin{equation*}
		\dv u_0 \in L^2(\Omega),
		\quad u_0 \cdot \n = 0 \text{ on } \del \Omega,
    \quad\text{and}
    \quad f\in L^2(0,T;V').
  \end{equation*}
  If either
  $$ 
    \prt\Omega \text{ is } C^2
  $$
  or
  $$
			\prt\Omega \text{ is locally Lipschitz}
      \text{ and }
      \divergence u_0 \in H^2(\Omega),
	$$
  then there exists a weak solution $u$ to~\eqsysWeakS with initial data $u_0$ such that
	$u \in C([0, T]; V')$,
  $\prt_t u \in L^{4/3}(0, T; V')$ and
  $\divergence u \in C^\infty(\Omega \times (0, T))$.
	In two dimensions, the exponent $4/3$ can be improved to $2$.

	If further $f \in L^1(0, T; H^{-1})$,
	there exists a distribution $p$ such that equation
	\begin{align}\label{eqnDistSol}
			\delt u + (u \cdot \grad) u - \nu \lap u + \grad p
					= f,
	\end{align}
	is satisfied in the sense of distributions.
\end{theorem}
\begin{remark}
  We reiterate that in Lipschitz domains we need an added regularity assumption on $\divergence u_0$.
  In $C^2$ domains, however, we can dispense with this assumption by using~\eqref{eqnweaklifting} (see also Remark~\ref{rmkOnlyC2Use}).
\end{remark}

The technique used to prove existence combined with relatively standard methods quickly yields uniqueness of weak solutions in 2D.

\begin{proposition}[Uniqueness in 2D]\label{ppn2DWeakUniqueness}
If $\Omega \subseteq \R^2$ is a bounded, connected Lipschitz domain then weak solutions to~\eqsysWeakS are unique.
\end{proposition}

For regular enough initial data it also yields strong solutions and higher regularity, which we discuss in Section~\ref{sxnStrong}.
We begin by proving Theorem~\ref{thmExistence}.

\begin{proof}[Proof of Theorem~\ref{thmExistence}]
  
Assume momentarily that $u$ is a weak solution of the \extended Navier-Stokes system as defined in \refD{WeakS}.
Let $g = \divergence u$, $g_0 = \dv u_0$, $z$ be the solution to~\eqref{eqnzq}, and $v = u - z$.
The main point of this decomposition is that equation~\eqref{eqnWeakSdivergence} completely determines $g$ in terms of $\divergence u_0$, which in turn determines $z$.

From~\refE{WeakSdivergence}, $g$ is the (weak) solution of
\begin{equation}\label{eqngHeat}
	\begin{beqn}
		\prt_t g = \nu \Delta g & in $\Omega \times (0, T]$, \\
		\grad g \cdot \n = 0 & on  $\prt \Omega \times (0, T]$, \\
		g(0) = g_0 & on $\Omega$.
	\end{beqn}
\end{equation}
By \refL{HeatEqtildeHeat}, $g \in C( [0, T]; L^2) \cap L^2(0, T; H^1)$.

We observe that by testing with the constant function $q\equiv 1$, we have
  \[
    \dt \int_\Omega g
        = \dt (g, q)
        = - \nu (\grad g, \grad q)
        = 0.
\] 
Since $g(0) = \dv u_0$, we have $\int_\Omega g(0) = \int_{\prt \Omega} u_0 \cdot \n = 0$, so $g(t)$ has mean zero for all $t \ge 0$. This allows us to apply \refL{TemamLemma} to lift $g$ to the unique vector field $z$ that solves~\eqref{eqnzq}.

By \refL{BasiczRegularityL2}, $z \in C( [0, T]; H_0^1) \cap L^2(0, T; H^2)$. Hence, $z$ has sufficient regularity and all we need to prove is that there exists a
$v \in C( [0, T]; V') \cap L^2(0, T; V)$
satisfying the perturbed weak-form Navier-Stokes equation,
\begin{equation}\label{eqnENSEquiv}
	\dt (v,\varphi) + \nu( \grad v,\grad\varphi) +(v \cdot \grad v,\varphi) 
		= \langle\widetilde{f}, \varphi\rangle - (v \cdot \grad z,\varphi) 
			- (z \cdot \grad v,\varphi),
\end{equation}
for any test function $\varphi \in V$, with initial data $v_0 = u_0 - z(0)$. Here, 
\begin{equation}\label{eqntildef}
\widetilde{f} \defeq f - z \cdot \grad z - \prt_t z,
\end{equation}
 and we took advantage of $(\grad z, \grad \varphi) = - (\Delta z, \varphi) = 0$ by \refE{zq}.
Note that the assumption $u_0 \cdot \n=0$ on $\del \Omega$ guarantees $v_0 \in H$.
  
We prove the existence of a function~$v$ satisfying~\eqref{eqnENSEquiv} following the classical approach (e.g.\ \cites{bblLeray,bblConstFoias,bblTemam}).
Since the standard Galerkin scheme (\cites{bblConstFoias, bblTemam}, for instance) readily yields an approximating sequence of solutions to~\eqref{eqnENSEquiv}, we only prove an a~priori estimate.
Choosing $\varphi = v$, equation~\eqref{eqnENSEquiv} becomes
  \begin{equation}\label{eqnENSnormV}
    \frac{1}{2} \dt \norm{v}_{L^2}^2 + \nu \norm{\grad v}_{L^2}^2
      = \langle\widetilde{f}, v\rangle - (v \cdot \grad v, v) - (v \cdot \grad z, v) - (z \cdot \grad v, v).
  \end{equation}
Two of the terms on the right-hand side of \refE{ENSnormV} are readily dealt with:
  $$
    \abs{\dualpairing{\widetilde{f}}{v}} \leq \norm{v}_{V} \norm{\widetilde{f}}_{V'} \leq \frac{\nu}{2} \norm{\grad v}_{L^2}^2 + \frac{C}{\nu} \norm{\widetilde{f}}_{V'}^2
  $$
  and
  $$
    \ip{ (v \cdot \grad) v}{ v } = 0.
  $$
  We split the analysis of the remaining terms into two cases depending on the dimension.
  If $d = 2$, using the Sobolev,  H\"older and Ladyzhenskaya  inequalities (see for instance~\cite{bblTemam}*{Lemma III.3.3}) yield
  $$
    \abs{\ip{(z \cdot \grad) v}{v}}
      = \frac{1}{2} \abs*{ \ip{\divergence z}{\abs{v}^2} }
      \leq \frac{1}{2} \norm{\grad z}_{L^2} \norm{v}_{L^4}^2
      \leq \frac{\nu}{8} \norm{\grad v}_{L^2}^2 + \frac{C}{\nu} \norm{\grad z}_{L^2}^2 \norm{v}_{L^2}^2
  $$
  The term $\ip{ (v \cdot \grad) z}{v}$ satisfies the same bound.
  Consequently, equation~\eqref{eqnENSnormV} reduces to the differential inequality,
  $$
    \dt \norm{v}_{L^2}^2 + \frac{\nu}{2} \norm{\grad v}_{L^2}^2 \leq \frac{c}{\nu} \norm{\widetilde{f}}_{V'}^2 + \frac{C}{\nu} \norm{\grad z}_{L^2}^2 \norm{v}_{L^2}^2,
  $$
  and Gronwall's lemma gives
  \begin{multline}\label{eqnEnergyBoundWider}
    \norm{v(t)}_{L^2}^2 + \nu \int_0^t \norm{\grad v(s)}_{L^2}^2 \, ds\\
      \leq \left( \norm{v_{0}}_{L^2}^2 +\frac{C}{\nu} \int_0^t \norm{\widetilde{f}(s)}_{V'}^2 \, ds \right) \exp \left( \frac{C}{\nu} \int_0^t \norm{\grad z(s)}_{L^2}^2 \, ds \right).
  \end{multline}

  The three-dimensional case is similar, except that we use now the interpolation inequality $\norm{v}_{L^4}\leq \norm{v}_{L^2}^{1/4}\norm{v}_{L^6}^{3/4}$:
  $$
    \abs{\ip{(z \cdot \grad) v}{v}}
      = \frac{1}{2} \abs*{ \ip{\divergence z}{\abs{v}^2} }
      \leq \frac{1}{2} \norm{\grad z}_{L^2} \norm{v}_{L^4}^2
      \leq \frac{\nu}{8} \norm{\grad v}_{L^2}^2 + \frac{C}{\nu^3} \norm{\grad z}_{L^2}^4 \norm{v}_{L^2}^2.
  $$
  Hence instead of~\eqref{eqnEnergyBoundWider}, Gronwall's lemma gives
  \begin{multline}\label{eqnEnergyBoundWider3D}
    \norm{v(t)}_{L^2}^2 + \nu \int_0^t \norm{\grad v(s)}_{L^2}^2 \, ds\\
      \leq \left( \norm{v_{0}}_{L^2}^2 +\frac{C}{\nu} \int_0^t \norm{\widetilde{f}(s)}_{V'}^2 \, ds \right) \exp \left( \frac{C}{\nu^3} \int_0^t \norm{\grad z(s)}_{L^2}^4 \, ds \right).
  \end{multline}

To complete these bounds in $2$ or $3$ dimensions, we need to show that the right-hand sides of~\eqref{eqnEnergyBoundWider} and~\eqref{eqnEnergyBoundWider3D} remain bounded for all $t\in[0,T]$.     The existence of a weak solution then follows by standard methods. (We elaborate somewhat on these methods in \refS{Strong}, where we construct higher-regularity solutions.)

Because $\prt_t z$ appears in $\widetilde{f}$, bounding the right-hand sides of \refEAnd{EnergyBoundWider}{EnergyBoundWider3D} will require knowledge of the time regularity of $z$, something we have not so far needed. We divide the analysis into two cases.
  
\setcounter{CaseCounter}{0}
\begin{case}[$\dv u_0 \in L^2$ and $\prt \Omega$ is $C^2$]
 
From \refL{BasiczRegularityL2},
    \begin{equation}\label{weakzReg}
      z\in C( [0,T]; H_0^1)\cap L^2(0,T;H^2)
      \quad\text{and}\quad
      \prt_t z\text{ in }L^2(0,T;L^2).
    \end{equation}
That $\prt_t u \in L^1(0, T; V')$ then follows, since $L^2 \subseteq V'$.
In both $2$ and $3$ dimensions, duality and the interpolation inequality $\norm{g}_{L^4}\leq \norm{g}_{L^2}^{1/4}\norm{g}_{L^6}^{3/4}$ give
    $$
      \norm{(z \cdot \grad) z}_{V'} \leq \norm{\grad z}_{L^2}^2,
    $$
    and hence
		\begin{equation}\label{eqnTildeFBound}
      \norm{\widetilde{f}}_{V'}^2 \leq C \left( \norm{\prt_t z}_{L^2}^2 + \norm{\grad z}_{L^2}^2 + \norm{f}_{V'}^2 \right).
		\end{equation}
    Now~\eqref{weakzReg} shows that the right hand sides of both~\eqref{eqnEnergyBoundWider} and~\eqref{eqnEnergyBoundWider3D} remain bounded. 
  \end{case}

  \begin{case}[$\divergence u_0 \in H^2$]
    The proof is similar to the previous case, except that because the boundary is only Lipschitz continuous,
    we cannot use the bound in \refE{zBoundL2} on $\prt_t z$.
    Instead, observe that $\prt_t g$ satisfies the heat equation with an initial value of $\nu \lap \divergence u_0$, which lies in $L^2(\Omega)$.
    Hence both $g \in C( [ 0, T]; L^2)$ and $\prt_t g \in C( [ 0, T]; L^2 )$.
    Consequently, instead of~\eqref{weakzReg}, \refL{TimeRegularityz} applied to $g$ and $\prt_t g$ gives
    $$
      z \in L^\infty( 0, T; H^1 )
      \quad\text{and}\quad
      \prt_t z \in L^\infty( 0, T; H^1 ),
    $$
    and the same argument used in the previous case shows that the right hand sides of both~\eqref{eqnEnergyBoundWider} and~\eqref{eqnEnergyBoundWider3D} remain bounded. 
  \end{case}
  
	We now turn to time regularity of $u$.
	Integrating~\eqref{eqnWeakSu} by parts gives
	\begin{equation}\label{eqnDtVWeak}
				\frac{d}{dt} \dualpairing{u}{\varphi} = \dualpairing{F}{\varphi}
				\text{ for all } \varphi \in V,
				\quad\text{where }
				F \defeq f - u \cdot \grad u + \nu \Delta u.
	\end{equation}
	Observe $F \in L^{4/3}(0, T; V')$.
	Indeed, $f \in L^2( 0, T; V' )$ by assumption.
	Since $u \in L^2(0, T; H^1_0 )$, we certainly have $\lap u \in L^2( 0, T; H^{-1} )$.
	Finally, for the nonlinear term, standard Calculus inequalities~\cite{bblTemam}*{Lemma III.3.3} give
	$$
		\abs*{\ip{(u\cdot \grad) u}{\phi}}
			\leq \norm{u}_{L^3} \norm{\grad u}_{L^2} \norm{\phi}_{L^6}
			\leq C \norm{u}_{L^2}^{1/2} \norm{\grad u}_{L^2}^{3/2} \norm{\phi}_{H^1}.
	$$
	for any $\phi \in H^1$.
	Thus, by duality, $\norm{u \cdot \grad u}_{H^{-1}} \leq C \norm{u}_{L^2}^{1/2} \norm{\grad u}_{L^2}^{3/2}$ and hence
	$u \cdot \grad u \in L^{4/3}(0, T; H^{-1})\subset L^{4/3}(0, T; V')$.
	This shows $F \in L^{4/3}(0, T; V')$.
	Now~\eqref{eqnDtVWeak} and the fact that $u, F \in L^{4/3}(0, T; V')$ imply that $\delt u = F$ (as elements of $V'$), and $u \in C(0, T; V')$ as desired (see for instance~\cite{bblTemam}*{Lemma III.1.1}).

	To establish~\eqref{eqnDistSol}, define
	$$
		G(t) = u(t) - u_0 + \int_0^t \big(  u \cdot \grad u - \nu \lap u + f \big) \, ds.
	$$
	Then $G \in L^1(0, T; H^{-1})$ and by equation~\eqref{eqnWeakSu} we have
	$$
		\dualpairing{G}{\varphi} = 0 \quad\text{for all } \varphi \in V.
	$$
	Consequently,
	the de Rham Lemma~\cite{bblTemam}*{Proposition I.1.1, Remark I.1.9} guarantees the existence of $P \in L^1(0, T; L^2)$ such that $G = \grad P$.
	Setting $p = -\delt P$ and taking the distributional time derivative of the equation $G = \grad P$ immediately gives~\eqref{eqnDistSol}.
\end{proof}

We conclude this section by proving uniqueness in 2D.
\begin{proof}[Proof of Proposition~\ref{ppn2DWeakUniqueness}]
  Suppose $u_1 = v_1 + z_1$ and $u_2 = v_2 + z_2$ are two solutions.
  Then $\dv u_1 = \dv u_2$ since they each solve the heat equation with the same initial data.
  By \refL{TemamLemma} it follows that $z_1 = z_2$.
  Rewriting the equations for~$v_1$, $v_2$ as~\eqref{eqnENSEquiv} and subtracting gives
  \begin{equation*}
    \dt (v,\varphi)
    + (v_1 \cdot \grad v + v \cdot \grad v_2 + v \cdot \grad z + z \cdot \grad v,\varphi)
    = -\nu (\grad v,\grad\varphi),
  \end{equation*}
  where $v=v_1-v_2$, and $z = z_1 = z_2$.
  Taking $\varphi=v$ and using Ladyzhenskaya's inequality  gives
  \begin{align*}
    \frac{1}{2} \dt \norm{v}_{L^2}^2 + \nu \norm{\grad v}_{L^2}^2
      &= - \ip{ (v \cdot\grad) v_2}{v} - \ip{ (v \cdot \grad) z}{v} - \ip{(z \cdot \grad) v}{v}\\
      &\leq c \Bigl(
				\begin{multlined}[t]
					\norm{\grad v_2}_{L^2} \norm{v}_{L^2} \norm{\grad v}_{L^2} + \norm{ \grad z}_{L^2} \norm{v}_{L^2}\norm{\grad v}_{L^2} \mathop+\\
						+ \norm{\divergence z}_{L^2} \norm{v}_{L^2}\norm{\grad v}_{L^2}\Bigr)
					\end{multlined}\\
      &\leq \frac{\nu}{2} \norm{\grad v}_{L^2}^2+ \frac{c}{\nu} \pr{\norm{\grad v_2}_{L^2}^2 + \norm{\grad z}_{L^2}^2 +  \norm{\dv u}_{L^2}^2}	\norm{v}_{L^2}^2.
  \end{align*}
  Integrating in time yields
  \begin{align*}
    \norm{v(t)}_{L^2}^2 + \nu \int_0^t \norm{\grad v(s)}_{L^2}^2 \, ds
      \leq \int_0^t \mu(s) \norm{v(s)}_{L^2}^2 \, ds,
  \end{align*}
  where $\mu(s) \defeq \frac{c}{\nu} ( \norm{\grad v_2}_{L^2}^2 + \norm{\grad z}_{L^2}^2 +  \norm{\dv u}_{L^2}^2)$.
  The energy estimate for the heat equation and~\eqref{eqnEnergyBoundWider} show that $\int_0^t \mu(s) \, ds < \infty$.
  Thus Gronwall's lemma applies and shows that $v \equiv 0$, yielding uniqueness.
\end{proof}

%
%
\section{Higher regularity in \texorpdfstring{$2D$}{2D}}\label{sxnStrong}
In \refT{HigherRegularity}, we obtain strong solutions by assuming more regularity on the initial data. 
Here we provide the details of the Galerkin approximation that were suppressed in the proof of 
\refT{Existence}.

\begin{theorem}\label{thmHigherRegularity}
  Let $\Omega\subset\mathbb{R}^2$ be a bounded $C^2$ domain, and suppose $u_0, f$ satisfy
  \begin{gather*}
    u_0\in H_0^1\cap H^2,
    \quad \dv u_0\in  H^2,
    \quad \grad \dv u_0 \cdot \n = 0,
    \quad \grad\lap (\dv u_0)\cdot \n=0\\
    f \in L^2(0,T;V'),\quad
    \prt_t f\in L^2(0,T;V'),
    \quad f(0)\in L^2.
  \end{gather*}
  If 
  $u$ is the (unique) solution to~\eqref{eqnENSSimpleForm} with initial data $u_0$, then
  $$
    \prt_t u\in L^\iny(0, T; L^2) \cap L^2(0, T; H_0^1).
  $$
  If further $f\in L^\iny(0,T;L^2)$, then $u\in L^\iny(0, T; H^2)$.
\end{theorem}
\begin{proof}
 As before, let $g = \divergence u$, let $z$ be the solution of~\eqref{eqnzq}, and let $v = u - z$.


Lemma~\ref{lmaBasiczRegularityH2} and  Remark~\ref{rmkEvenHighRegularity} guarantee that $z \in L^\iny(0, T; H^3)$ with $\prt_t z \in L^\iny(0, T; H^1) \cap L^2(0,T;H^2)$. Thus  in order to  show that $\prt_t u\in L^\iny(0, T; L^2) \cap L^2(0, T; H_0^1)$ we  only need to show that $\prt_t v \in L^\iny(0, T; H) \cap L^2(0, T; V)$. Furthermore we will show that when $f \in L^\iny(0,T;L^2)$ we have $v \in L^\iny(0, T; H^2)$.

  Let $w_k=w_k(x)$ for $k=1,2,\dots$ be the $L^2$-orthonormal eigenfunctions of the Stokes operator, $A \defeq -\LP\Delta$, with the eigenvalues $\lambda_1, \lambda_2 \dots$, respectively.
  For $u,v,w\in H_0^1(\Omega)$, define the trilinear form, $b(u,v,w) = \ip{u \cdot \grad v}{w}$, where $\ip{\cdot}{\cdot}$ denotes the 
  usual inner-product on $L^2(\Omega)$.
  For each $k\in\N^*$ define the approximate solution $v_k$ by
  \begin{align*}
    v_k \defeq \sum_{j=1}^kg_{jk}(t)w_j,
  \end{align*}
  where the coefficients $g_{jk}(t)$ are chosen so that $v_k$ solves
  \begin{equation}\label{eqnENSEquivGalerkin}
    (\prt_t v_k, w_j)  + \nu (Av_k,w_j) +b(v_k,v_k,w_j) 
      =\langle \widetilde{f} ,w_j\rangle 
    -b(v_k, z,w_j) - b(z, v_k,w_j)
  \end{equation}
  with initial data $v_k(0) = v_{0k} \defeq \sum_1^k g_{jk}(0) w_j$,
  $g_{jk}(0) = \ip{u_0}{w_j}$,
  and $\widetilde{f} \defeq f - z \cdot \grad z - \prt_t z$.

  This reduces to the nonlinear system of $k$ ODEs,
  \begin{multline}\label{eqnENSEquivGalerkinCoef}
    g_{jk}'(t)+\nu\lambda_jg_{jk}+\sum_{r=1}^k\sum_{s=1}^kb(w_r,w_s,w_j)g_{rk}g_{sk}\\
      =\langle\tilde{f},w_j\rangle-\sum_{r=1}^k b(w_r,z,w_j)g_{rk} -\sum_{r=1}^kb(z,w_r,w_j)g_{rk},
  \end{multline}
  with initial data $g_{jk}(0)=(v_{k}(0),w_j)$ for $j=1,\dots, k$.
  Standard ODE theory shows that~\eqref{eqnENSEquivGalerkinCoef} has a unique absolutely continuous solution, $g_{ij}(t)$, for $i,j \in \{1,\dots k\}$.

  Multiplying \eqref{eqnENSEquivGalerkinCoef} by $g_{jk}$ and summing for $j=1,\dots, k$ gives
  \begin{align*}
    \frac{1}{2} \dt \norm{v_k}_{L^2}^2
    + \nu \norm{\grad v_k}_{L^2}^2
    = \langle\widetilde{f}, v_k\rangle 
    - b(v_k, z, v_k) 
    - b(z, v_k, v_k),
  \end{align*}
  where we used $b(v_j,v_k,v_k)= 0$ and $(Av_k,v_k)=(\grad v_k,\grad v_k)$.

Multiplying \refE{ENSEquivGalerkin} by $g_{jk}'$ and summing for $j=1,\cdots, k$, we have
\begin{align*}
		\norm{v_k'}_{L^2}^2 + \nu(\Delta v_k, v_k') + b(v_k,v_k,v_k') 
			= \langle\widetilde{f},v_k'\rangle
				- b(v_k, z, v_k') 
				- b(z, v_k, v_k').
\end{align*}
Thus, at time $t=0$, we have 
\begin{align*}
		&\norm{v_k'(0)}_{L^2} \leq \nu\norm{\Delta v_k(0)}_{L^2} + \norm{B(v_k(0),v_k(0))}_{L^2} \\
			&\qquad\qquad+ \smallnorm{\widetilde{f}(0)}_{L^2}
				+ \norm{B(v_k(0), z(0))}_{L^2} + \norm{B(z(0), v_k(0))}_{L^2},
\end{align*}
where we used the
Cauchy-Schwarz inequality and the standard notation, $B(a,b) \defeq \LP(a\cdot\nabla b)$. Observe the terms on right side are bounded since $v_0\in H^2\cap V$, 
$f(0)\in L^2$, and the high enough  regularity of $z$ as given by Lemma~\ref{lmaBasiczRegularityH2} and Remark~\ref{rmkEvenHighRegularity}.

We now differentiate \eqref{eqnENSEquivGalerkin} in time, multiply by $g_{jk}'$ and sum for $j=1,\cdots, k$ to obtain

\begin{align*}
	\frac{1}{2} \dt& \norm{v_k'}_{L^2}^2
		+ \nu \norm{\grad v_k'}_{L^2}^2
		+b(v_k',v_k,v_k')
		+b(v_k,v_k',v_k') \\
		&= \langle\prt_t\widetilde{f}, v_k'\rangle
		- b(v_k',  z, v_k')
		-b(v_k, \prt_t z, v_k') \\
		&\qquad
		- b(\prt_t z, v_k, v_k')
		-b(z, v_k', v_k').
\end{align*}
Using Ladyzhenskaya's inequality, we have
\begin{align*}
	b(v_k',v_k,v_k')
		\leq \frac{\nu}{10} \norm{\grad v_k'}_{L^2}^2
		+C\nu^{-1}\norm{\grad v_k}_{L^2}^2\norm{v_k'}_{L^2}^2,
\end{align*}
and similar estimates hold for the other trilinear terms. Thus,
\begin{align*}
	&\frac{1}{2} \dt \norm{v_k'}_{L^2}^2
		+ \frac{\nu}{2} \norm{\grad v_k'}_{L^2}^2\\
		&\qquad
		\leq \langle \prt_t\widetilde{f}, v_k'\rangle
			+C\nu^{-1}\left(\norm{\grad v_k}_{L^2}^2
				+\norm{z}_{H^1}^2\right)\norm{v_k'}_{L^2}^2
			+C\nu^{-1}\norm{\prt_t z}_{H^1}^2\norm{\grad v_k}_{L^2}^2,
\end{align*}	
where we used Poincar\'e's inequality.

Recall that $\prt_t\tilde{f}=  \prt_tf-\prt_t z\cdot \grad z-z\cdot \grad \prt_t z-\prt_{tt}z$. By Remark~\ref{rmkEvenHighRegularity} in the Appendix, if $\nabla\cdot u_0\in H^2$, then $\prt_{tt}z\in L^2(0,T;L^2)$. Hence,
\begin{align*}
	|\langle \prt_t\widetilde{f}, v_k'\rangle|
		&\leq\frac{\nu}{4} \norm{\grad v_k'}_{L^2}^2 
		+ C\nu^{-1}\norm{f_t}_{H^{-1}}^2 
		+ C\nu^{-1}\norm{z}_{H^1}^2\norm{\prt_t z}_{H^1}^2 
		+ C\nu^{-1}\norm{\prt_{tt}z}_{L^2}^2,
\end{align*}	
where we also used Poincar\'e inequality. 
We arrive at the a~priori estimate,
\begin{align*}
	& \dt \norm{v_k'}_{L^2}^2
		+ \nu \norm{\grad v_k'}_{L^2}^2
		\leq C\nu^{-1}\left(\norm{\grad v_k}_{L^2}^2
				+\norm{z}_{H^1}^2\right)\norm{v_k'}_{L^2}^2
		+C\nu^{-1}\phi(t),		
\end{align*}	
where
$\phi(t)=\norm{\prt_t z}_{H^1}^2\norm{\grad v_k}_{L^2}^2
		+\norm{\prt_t f}_{H^{-1}}^2 
		+\norm{z}_{H^1}^2\norm{\prt_t z}_{H^1}^2
		+\norm{\prt_{tt}z}_{L^2}^2$.
Therefore, by Gronwall's inequality,
\begin{align*}
	&\norm{v_k'(t)}_{L^2}^2
		+\nu\int_0^t\norm{\grad v_k'(s)}_{L^2}^2\,ds
		\leq\left(\norm{v_k'(0)}_{L^2}^2+C\nu^{-1}\int_0^t\phi(s)\,ds\right) \\
		&\qquad\qquad\qquad\qquad\qquad
		\times\exp\left(C\nu^{-1}\int_0^t \left(\norm{\grad v_k(s)}_{L^2}^2
				+\norm{z(s)}_{H^1}^2\right)\,ds\right).
\end{align*}
Hence, $v_k' \in L^\iny(0, T; H) \cap L^2(0, T; V)$ and the first part of the assertion is established. 

Now, assume additionally that $f\in L^\iny(0,T;L^2)$. It follows from \refE{ENSEquiv} that

\begin{align}\label{eqnStokesRegNonlinearity}
	\nu(\grad v(t),\grad \bar{v})=(h(t),\bar{v})
\end{align}
for all $\bar{v}\in V$, where 
\begin{align*}
	h(t)=-\prt_tv
		-\LP(v\cdot \grad v)
		+\LP\tilde{f}
		-\LP (v\cdot \grad z)
		-\LP (z\cdot \grad v). 
\end{align*}
Since $\LP \bar{v} = \bar{v}$ and $\LP$ is self-adjoint, we have
\begin{align}\label{eqnLPvvvbarBound}
	\begin{split}
		b(v(t), v(t)), \bar{v})
			&\leq C\norm{v(t)}_{L^4}
				\norm{\grad v(t)}_{L^2}
				\norm{\bar{v}}_{L^4}
			\leq C \norm{\grad v}_{L^2}^2\norm{\bar{v}}_{L^4} \\
			&\leq C \norm{\bar{v}}_{L^4}.
		\end{split}
\end{align}
Here we used $v\in L^\iny(0,T;V)$, which follows from the fact that $v, \delt v \in L^2( 0, T; V )$.

Following Galdi \cite{bblGaldi}, let $\mathcal{H}_4(\Omega)$ be the completion in the $L^4$-norm of the divergence free vector fields in $C_c^\iny(\Omega)$ and
\begin{align*}
	\mathcal{G}_4(\Omega)
		= \set{u \in L^4(\Omega) \colon u
		= \grad p \text{ for some } p \in W^{1, 4}(\Omega)}.
\end{align*}
Now let $w$ be any vector field in $L^4(\Omega)$. By the Helmholtz-Leray decomposition of $L^4$ (see Equation III.1.5 and Remark III.1.1 of \cite{bblGaldi}), $w = \bar{\bar{v}} + \grad q$ for some $\bar{\bar{v}}$ in $\mathcal{H}_4(\Omega)$ and $\grad p$ in $\mathcal{G}_4(\Omega)$, with $\norm{\bar{\bar{v}} }_{L^4} \leq C \norm{w}_{L^4}$. Since $\Omega$ is bounded, it follows that $\bar{\bar{v}} = \LP w$ and that $\grad q$ lies in $L^2(\Omega)$.

Then by \refE{LPvvvbarBound},
\begin{align*}
	b(v(t), v(t), w)
		&= b(v(t), v(t), \bar{\bar{v}})
		\leq C \norm{\bar{\bar{v}} }_{L^4}
		\leq C \norm{w}_{L^4}.
\end{align*}

What this shows is that $\LP (v\cdot \grad v)\in L^\iny(0,T;L^{4/3})$.
Similarly, we get that $\LP (v\cdot \grad z)$ and $\LP (z\cdot \grad v)$ are in $L^\iny(0,T;L^{4/3})$. Also, we have $\LP \tilde{f}-\prt_t v \in L^\iny(0,T;L^2)$. Hence, $h(t)\in L^\iny(0,T;L^{4/3})$. Then by Proposition I.2.2 of \cite{bblTemam} applied to the Stokes problem \eqref{eqnStokesRegNonlinearity} and the Sobolev embedding theorem we conclude that $v\in L^\iny(0,T;L^\iny)$. 

Now, we can use, in place of \refE{LPvvvbarBound}, the bound,
\begin{align*}
	b(v(t), v(t),\bar{v})
		&\leq C\norm{v(t)}_{L^\iny}
			\norm{\grad v(t)}_{L^2}
			\norm{\bar{v}}_{L^2},
\end{align*}  
which implies that $\LP (v\cdot \grad v)\in L^\iny(0,T;L^2)$, and so $h(t)\in L^\iny(0,T;L^2)$. Another application of Proposition I.2.2 in \cite{bblTemam}  to \eqref{eqnStokesRegNonlinearity} leads to $v \in L^\iny(0,T;H^2)$, finishing the proof.
\end{proof}
\begin{remark*}
If we assume that $\Omega$ is of class $C^\iny$ and $f\in C^\iny(\Omega\times [0,T])$, then the solution $u$ is in $C^\iny(\Omega\times (0,T])$ provided that the data satisfies suitable compatibility conditions. Also, one can obtain more regular solutions in the 3D case if the given data is, in addition, sufficiently small. This can all be accomplished along the lines laid out by Temam in Remarks III.3.7 and III.3.8 and Theorems III.3.7 and III.3.8 of \cite{bblTemam}.
\end{remark*}

\begin{proposition}\label{ppnStrongHoldsae}
	Let $u$ be a strong solution to the extended Navier-Stokes equations given by \refT{HigherRegularity}
	with $f \in L^\iny(0, T; L^2)$.
	Then~$u \in C( [ 0, T ]; H^1_0)$ and satisfies~\eqref{eqnENSSimpleForm} as distributions and almost everywhere on $[0, T] \times \Omega$.
\end{proposition}
\begin{proof}
	That $u \in C( [ 0, T ]; H^1_0)$ follows directly from $u\in H^1(0, T; H^1_0)$, applying, for instance,
	Theorem 2 of \cite{bblEvansBook}*{\S5.9.2}.
	To see~$u$ satisfies~\eqref{eqnENSSimpleForm} almost everywhere, Theorem~\ref{thmHigherRegularity} gives $u \in H^2$ for almost all $t \in [0, T]$.
	Using standard estimates this implies $(u \cdot \grad) u \in L^2$, and consequently both $\lhp (u \cdot \grad) u$ and $\lhp \lap u$ are defined.
	Thus, setting $U = \prt_t u + \LP (u \cdot \grad u-f) - \nu(\LP \Delta u + \grad \dv u)$, equation~\eqref{eqnWeakSu} forces $\ip{U}{\varphi} = 0$ for all $\varphi \in V$.
	Further, equation~\eqref{eqnWeakSdivergence} gives $\ip{U}{\grad q} = 0$ for all $q \in H^1$.
	Since $U \in L^2(\Omega)$ (by Theorem~\ref{thmHigherRegularity}) this implies $U = 0$ almost everywhere, and that~\eqref{eqnENSSimpleForm} is satisfied as distributions.
\end{proof}

%
%
\section{Well-posedness for the Shirokoff-Rosales system}\label{sxnSR}

In this section we consider the pressure-Poisson system in~\cite{bblShirokoffRosales}*{equations (20), (A.4), and Appendix~A}, which was introduced to provide a high-order, efficient time-discrete scheme for the incompressible Navier-Stokes equations in irregular domains.
The formal limit of their time discrete scheme satisfies the equations
\begin{equation}\label{eqnSRu}
  \begin{beqn}
    \prt_t u + (u \cdot \grad) u - \nu \lap u + \grad p = f & in $\Omega$,\\
    u \cross \n = 0 & on $\prt \Omega$,\\
    \divergence u = 0 & on $\prt \Omega$,
  \end{beqn}
\end{equation}
and
\begin{equation}\label{eqnSRp}
  \begin{beqn}
    \lap p =-\divergence\left((u\cdot\nabla)u\right)+ \divergence f & in $\Omega$,\\
    \grad p \cdot \n = \left( \nu \lap u-(u\cdot\nabla) u + \lambda u + f \right) \cdot \n - \mathcal C & on $\prt \Omega$.
  \end{beqn}
\end{equation}
Here $f$ is the external forcing, $\nu > 0$ is the viscosity, $\lambda > 0$ is an artificial damping parameter, and $\mathcal C = \mathcal C(t)$ is defined by
\begin{equation}\label{eqnSRC}
  \mathcal C
    = \frac{1}{\abs{\prt \Omega}} \int_{\prt \Omega} (\nu \lap u + \lambda u) \cdot \n
    = \frac{1}{\abs{\prt \Omega}} \int_{\Omega} \bigl( \nu \lap \divergence u + \lambda \divergence u \bigr),
\end{equation}
which is exactly the compatibility condition required to solve~\eqref{eqnSRp}.
A similar system (with an additional boundary condition) appeared in~\cite{bblJohnstonLiu}*{\S2}.
Our aim is to study the well-posedness of these equations.

We begin by observing that the system~\eqsysSR formally reduces to the Navier-Stokes equations if the initial data $u_0$ satisfies the compatibility conditions $\divergence u_0 = 0$ in $\Omega$, and $u_0 \cdot \n = 0$ on $\prt \Omega$.
To see this, note that the evolution for $u$ in~\eqref{eqnSRu}\textsubscript{1} is the same as that in the Navier-Stokes equations.
Further, equation~\eqref{eqnSRu}\textsubscript{2} gives the tangential no-slip boundary conditions.
What we are missing, however, is the incompressibility constraint and the normal boundary condition.

First, to recover the incompressibility constraint, take the divergence of~\eqref{eqnSRu} and use~\eqref{eqnSRp}. This yields
\begin{equation}\label{eqnSRdivergence}
  \begin{beqn}
    \prt_t \divergence u - \nu \lap \divergence u = 0 & in $\Omega$,\\
    \divergence u = 0 & on $\prt \Omega$.
  \end{beqn}
\end{equation}
Thus if $\divergence u_0 = 0$, then for all $t \geq 0$ we must have $\divergence u = 0$ identically in $\Omega$, not just on $\prt \Omega$, recovering the incompressibility constraint.

Next, to recover the normal boundary condition, formally apply~\eqref{eqnSRu} for $u \cdot \n$ on $\prt \Omega$.
Combined with the boundary condition for $p$ in~\eqref{eqnSRp}, this yields
\begin{equation}\label{eqnSRuNormal}
	\prt_t (u \cdot \n) + \lambda u \cdot \n = \mathcal C \quad\text{on }\prt \Omega.
\end{equation}
Now if $\divergence u_0 = 0$, then the above argument shows $\divergence u = 0$ identically, and equation~\eqref{eqnSRC} forces $\mathcal C=0$.
Thus assuming initially $u_0 \cdot \n = 0$ on $\prt \Omega$, equation~\eqref{eqnSRuNormal} will imply $u \cdot \n = 0$ on $\prt \Omega$ for all $t > 0$, which recovers the missing normal boundary condition.
\medskip

We begin our study of well-posedness to~\eqsysSR by defining weak solutions.
Following the approach used for the \extended Navier-Stokes system of Definition~\ref{dfnWeakS}, we define weak solutions to~\eqsysSR by testing the divergence free part, and the divergence separately.
In addition, however, we must impose \eqref{eqnSRuNormal}, which is an automatic consequence of~\eqsysSR for smooth enough solutions.

\begin{definition}\label{dfnSRWeakS}
  We define a {\it weak solution} of~\eqsysSR to be a function $u$ such that 
	\begin{gather*}
    u\in L^\iny(0, T; L^2) \cap L^2(0, T; H^1),\\
    \divergence u\in L^\infty(0,T;L^2)\cap L^2(0,T;H^1_0)\cap L^1(0, T; W^{2,1} )
	\end{gather*}
  and equations~\eqref{eqnWeakSu}, \eqref{eqnWeakSdivergence} hold for almost all $t \in (0, T)$, every test function $\varphi\in V$, and every test function $q\in H^1_0$.
Further, on $\prt\Omega$, $u$ satisfies the boundary condition~\eqref{eqnSRu}\textsubscript{2} and the ODE~\eqref{eqnSRuNormal}, where $\mathcal C$ is given by~\eqref{eqnSRC}.
\end{definition}


The global existence of weak solutions to~\eqsysSR can be proved in a manner similar to Theorem~\ref{thmExistence}.
For $C^2$ domains, the results and proof exactly parallel Theorem~\ref{thmExistence}, and we address this in Remark~\ref{rmkC2Domains} below.
We are not able to treat arbitrary Lipschitz domains, however, and 
require an additional assumption on the regularity of the domain; we present the details below.
We also remark that one can use the same methods to prove regularity of weak solutions to~\eqsysSR analogous to~Theorem~\ref{thmHigherRegularity}.
\begin{theorem}\label{thmSRExistence}
  For $d = 2, 3$, let $\Omega \subset \R^d$, be a bounded Lipschitz domain such that $\n \in H^{1/2}(\del \Omega)$.
	Assume
  $$
    u_0\in L^2(\Omega),\quad \,\divergence u_0\in H^2(\Omega), \quad
    (u_0\cdot \n) \n \in H^{1/2}(\prt\Omega),
    \quad\text{and}\quad
    f\in L^2(0,T;V').
  $$
  There exists a weak solution, $u$, to~\eqsysSR with initial data $u_0$ such that $\divergence u \in C^\infty( \Omega \times (0,T) )$ 
  and this solution is unique if $d = 2$. 
  \end{theorem}
	Before presenting the proof, we remark that the assumption
	$\n \in H^{1/2}(\del \Omega)$
	is not satisfied by polygonal domains.
\begin{proof}[Proof of Theorem~\ref{thmSRExistence}]

Observe first that the divergence of a weak solution can be directly determined from the initial data by solving the heat equation.
Once this is known, the normal component of the weak solution on $\del \Omega$ can be determined using~\eqref{eqnSRuNormal} and~\eqref{eqnSRC}.
As before, our main idea is to combine the divergence and normal component into a solution of a stationary Stokes problem, and treat what remains as a perturbed Navier-Stokes equations.

To follow this plan, we let $g$ be a solution of the heat equation
\begin{equation}\label{eqnSRgHeat}
	\begin{beqn}
		\prt_t g = \nu \Delta g & in $\Omega \times (0, T]$,\\
		g = 0 & on  $\prt \Omega \times (0, T]$
	\end{beqn}
\end{equation}
with initial data $g(x, 0) = \divergence u_0(x)$.
Since
$\divergence u_0 \in H^2$,
regularity for the heat equation in Lipschitz domains (see for instance \cite{bblVrabie}*{p.$156$}) gives $g \in C_b(0, \infty; H^1)$.
Further, $\lap g$ also satisfies the heat equation with initial conditions $\lap \divergence u_0 \in L^2$, and hence $\lap g \in C_b(0, \infty; L^2)$.
Thus
\begin{align*}
	\bar{\mathcal C}
		\defeq \frac{1}{\abs{\prt \Omega}} \int_{\Omega} \bigl( \nu \lap g
			+ \lambda g \bigr)
\end{align*}
is a bounded continuous function of time.
Using this, define $h$ to be the solution to the ODE
\begin{equation}\label{eqnSRh}
	\prt_t h(x) + \lambda h(x) = \bar{\mathcal C} 
	\qquad
	\text{with initial data } h(x, 0) = u_0(x) \cdot \n,
\end{equation}
where $x \in \del \Omega$ is only a spatial parameter.

Our aim, naturally, will be to construct a weak solution, $u$, so that $\dv u = g$ and $u \cdot \n = h$ on $\prt \Omega$.
To do this, we combine both the divergence and the inhomogeneous boundary values of $u$ into a function~$z$, defined to be the solution of the (inhomogeneous) Stokes problem
\begin{gather}
	\begin{beqn}\label{eqnSRz}
		-\lap z + \grad q = 0 & in $\Omega$,\\
		\divergence z = g & in $\Omega$,\\
		z  = h \, \n & on $\prt \Omega$,
	\end{beqn}
\end{gather}
Lemma~\ref{lmaGaldi} will guarantee the existence of~$z \in H^1(\Omega)$, provided~$g \in L^2(\Omega)$, $h \, \n \in H^{1/2}(\del \Omega)$ and the solvability condition
\begin{equation}\label{eqnSRsolvability}
	\int_\Omega g = \int_{\del \Omega} h
\end{equation}
is satisfied.

The requirement $g \in L^2(\Omega)$ has been established above.
Our assumptions that both $\n$ and $(u_0 \cdot \n) \n$ are in $H^{1/2}$ and the ODE~\eqref{eqnSRh} will show that $h \, \n \in H^{1/2}( \del \Omega )$.
For the solvability condition~\eqref{eqnSRsolvability}, observe that equations~\eqref{eqnSRh} and~\eqref{eqnSRgHeat} imply
$$
	\dt \left( \int_{\del \Omega} h - \int_\Omega g \right) + \lambda\left( \int_{\del \Omega} h - \int_\Omega g \right) = 0.
$$
Since by definition $\int_\Omega g_0 = \int_\Omega \divergence u_0 = \int_{\del \Omega} u_0 \cdot \n = \int_{\del\Omega} h_0$, we must have~\eqref{eqnSRsolvability} satisfied for all $t \geq 0$.

Now following the proof of Theorem~\ref{thmExistence}, we define~$v\in V$ to be a weak solution of~\eqref{eqnENSEquiv} with initial conditions $v_0 = u_0 - z_0$.
Global existence of~$v$ will follow from the a~priori estimate~\eqref{eqnEnergyBoundWider} in 2D (or~\eqref{eqnEnergyBoundWider3D} in 3D), provided the right hand side is finite.
From~\eqref{eqnTildeFBound}, we see that this will follow if we show $z \in L^2( 0, T; H^1 )$ and $\delt z \in L^2( 0, T; L^2)$.
By \refL{Galdi}
it suffices to show 
\begin{equation}\label{reggh}
\mbox{$g , \delt g \in L^2( 0, T; L^2(\Omega))$ \quad and\quad  $h \n, \delt h \n \in L^2( 0, T; H^{1/2}(\del \Omega) )$.}
\end{equation}

From~\eqref{eqnSRgHeat} and the assumption $\divergence u_0 \in H^2(\Omega)$ it immediately follows that $g , \delt g \in L^2( 0, T; L^2(\Omega))$ as desired.
For the regularity of $h \n$, observe that Duhamel's principle gives
\begin{equation}\label{eqnDuhamel}
	h(t, x)\n(x)
		= e^{-\la t} (u_0(x) \cdot \n(x))\n(x) + \left(\int_0^t e^{-\la(t - s)} \bar{\mathcal C}(s) \, ds\right)\n(x)
\end{equation}
Since we already know that $\bar{\mathcal C}$ is a bounded continuous, and $\n \in H^{1/2}( \del \Omega)$ by assumption, we must have $h\n \in C_b(0, \infty; H^{1/2}(\del \Omega) )$.
Since $\delt h = -\lambda h + \bar{\mathcal C}$, this also implies $(\delt h)\n \in C_b(0, \infty; H^{1/2}(\del \Omega) )$.
Thus $h \n, \delt h \n \in L^2( 0, T; H^{1/2}(\del \Omega) )$ for any $T < \infty$.
This will show the global existence of the function~$v$.

To finish the proof, we define $u = v + z$.
Since~$\divergence u = \divergence z = g$ in $\Omega$, and $u \cdot \n = z \cdot \n = h$ on $\del \Omega$, we must have $\bar{\mathcal C} = \mathcal C$.
Now~\eqref{eqnENSEquiv} and~\eqref{eqnSRz} immediately imply that $u$ is the desired weak solution to~\eqsysSR.
Uniqueness in two dimensions follows using the same argument as in Theorem~\ref{thmExistence}.
\end{proof}

\begin{remark}[Existence in $C^2$ domains]\label{rmkC2Domains}
	As with Theorem~\ref{thmExistence}, we can prove existence in~$C^2$ domains with a reduced regularity assumption on the initial divergence: namely, we only need $\divergence u_0 \in L^2$.
	The main difference in this case is in making sense of~$\mathcal C$, which appears to require $\divergence u \in W^{2,1}$.
	The reason we don't need~$\divergence u \in W^{2,1}$ is because in the proof of Theorem~\ref{thmSRExistence}, we only use $\mathcal C$ (which is the same as $\bar{\mathcal C}$) to determine~$h$.
	Under the reduced regularity assumption~$\divergence u_0 \in L^2$, one can determine~$h$ by using~\eqref{eqnDuhamel} because
\begin{align*}
	\int_0^t e^{-\la(t - s)} \bar{\mathcal C}(s) \, ds
		&= \frac{1}{\abs{\prt \Omega}} \int_{\Omega} \int_0^t e^{-\la(t - s)} 
			\bigl( \nu \lap g
				+ \lambda g \bigr) \, ds \, dx \\
		&= \frac{1}{\abs{\prt \Omega}} \int_{\Omega} \int_0^t e^{-\la(t - s)} 
			\bigl( \prt_s g
				+ \lambda g \bigr) \, ds \, dx\\
		&= \frac{1}{\abs{\prt \Omega}} \int_\Omega
			\brac{g(x, t) - e^{-\la t} \dv u_0(x)} \, dx,
\end{align*}
which is certainly defined for $\dv u_0 \in L^2$.
With this, proving existence in a $C^2$ domain with only $L^2$ initial divergence is similar to Theorems~\ref{thmExistence} and~\ref{thmSRExistence}.
\end{remark}

%
%
\appendix
\section{Some estimates on the heat equation}\label{sxnHeat}
We now consider regularity results for solutions to the heat equation with Neumann boundary conditions. We begin with a basic fact:

\begin{lemma}\label{lmaHeatEqtildeHeat}
Let $\Omega$ be a Lipschitz domain and let $g_0\in L^2$. There exists a unique $g$ in $C( [0, T]; L^2)\cap L^2(0, T; H^1)$ with $\prt_t g$ in $L^2(0, T;\tilde H^{-1})$ satisfying~\eqref{eqngHeat}.
Moreover,
\begin{align}\label{eqnHeatRegBound}
	\begin{split}
	\norm{g}_{L^\iny(0, T; L^2)}
		&\leq \norm{g_0}_{L^2}, \\
	\norm{\grad g}_{L^2(0, T; L^2)}
		&\leq (2 \nu)^{-{1/2}} \norm{g_0}_{L^2}, \\
	\norm{\prt_t g}_{L^2(0, T;\tilde H^{-1})}
		&\leq (\nu/2)^{1/2} \norm{g_0}_{L^2},
	\end{split}
\end{align} where $\tilde H^{-1}(\Omega)$ denotes the dual of $H^1(\Omega)$.
\end{lemma}
\begin{proof}
  The existence of $g$ in $C( [0, T]; L^2)\cap L^2(0, T; H^1)$ and the first two bounds in
  \refE{HeatRegBound} are classical.
  To prove $\refE{HeatRegBound}_3$, choose any $\varphi$ in $H^1(\Omega)$.
  Using the Neumann boundary conditions on $g$ we see
  \begin{align*}
    \int_\Omega \prt_t g \, \varphi \,dx
    = \nu \int_\Omega\Delta g \, \varphi
    = -\nu \int_\Omega \nabla g \cdot \nabla\varphi.
  \end{align*}
  Hence
  \begin{align*}
    \norm{\prt_t g}_{\tilde{H}^{-1}}
    \leq \nu \norm{\grad g}_{L^2} \sup_{\norm{\varphi}_{H^1} = 1} \norm{\grad \varphi}_{L^2}
    \leq \nu \norm{\grad g}_{L^2}.
  \end{align*}
The estimate $\refE{HeatRegBound}_3$ now follows from $\refE{HeatRegBound}_2$.
\end{proof}

In \refL{BasiczRegularityL2}, we use \refL{HeatEqtildeHeat} to obtain basic estimates for the lifting, $z$, that we defined in Section~\ref{sxnDecomposition}.

\begin{lemma}\label{lmaBasiczRegularityL2}
Let $\Omega$ be a Lipschitz domain.
Assume that $g_0$ lies in $L^2$, let $g$ be the unique solution to \refE{gHeat}, and let $(z, p)$ solve \refE{zq} ($z$ being unique). Then $z$ lies in $C( [0, T]; H_0^1) \cap L^2(0, T; H^2)$
    \begin{align}\label{eqnzBoundL2Lip}
	\begin{split}
	    \sqrt{\nu} \norm{z}_{L^\iny(0, T; H^1)}
		+ \nu \norm{z}_{L^2(0, T; H^2)}
		&\leq C \sqrt{\nu} \norm{g_0}_{L^2}.
	\end{split}
    \end{align}
If $\prt \Omega$ is $C^2$ then $\prt_t z \in L^2(0, T; L^2)$ and
    \begin{align}\label{eqnzBoundL2}
	\begin{split}
	    \sqrt{\nu} \norm{z}_{L^\iny(0, T; H^1)}
		+ \norm{\prt_t z}_{L^2(0, T; L^2)}
		+ \nu \norm{z}_{L^2(0, T; H^2)}
		&\leq C \sqrt{\nu} \norm{g_0}_{L^2}.
	\end{split}
    \end{align}
\end{lemma}
\begin{proof}
	We apply \refL{TemamLemma} to \refL{HeatEqtildeHeat}, noting that we can only apply
	\refE{weaklifting} when $\prt \Omega$ is $C^2$.
\end{proof}

A key point in Lemma~\ref{lmaHeatEqtildeHeat} is that we only assume that $g_0\in L^2$. In order to obtain more regularity for the solution of \eqref{eqngHeat} we impose more regularity on the domain and the initial data.

The next classical result addresses the higher regularity for the heat equation with Neumann boundary conditions (see for instance \cite{bblPazy}):

\begin{lemma}\label{lmaHigherHeatRegularity}
Let $\Omega$ be a  $C^2$ domain. We assume that $g_0 \in H^2$ with $\frac{\partial g_0}{\partial\n}=0$ on $\prt\Omega$.
Then, for $g$ a solution of \eqref{eqngHeat} we have:
\begin{equation}
\sqrt{\nu} \norm{g}_{L^\infty(0,T;H^2)}+\norm{\prt_t g}_{L^2(0,T;H^1)}+\nu\norm{g}_{L^2(0,T;H^3)}
 \leq C\sqrt{\nu}\norm{g_0}_{H^2}
 \end{equation}
where the constant $C$ depends only on $\Omega$.
\end{lemma}

Similarly as before the lemma can be used to provide regularity results on the lifting, $z$:

\begin{lemma}\label{lmaBasiczRegularityH2}
    Let $\Omega$ be a  $C^2$ domain. Assume that $g_0$ lies in $H^2$ with 
    $\frac{\partial g_0}{\partial n}=0$, let $g$ be the unique solution to \refE{gHeat}, and let
    $(z, p)$ solve \refE{zq} ($z$ being unique). Then $z$ lies in $L^\iny(0, T; H^3) \cap L^2(0, T; H^4)$
    with
    $\prt_t z$ in $L^2(0, T; H^2)$, and
    \begin{align}\label{eqnzBoundH2}
	\begin{split}
	    \sqrt{\nu} \norm{z}_{L^\iny(0, T; H^3)}
		+ \norm{\prt_t z}_{L^2(0, T; H^2)}
		+ \nu \norm{z}_{L^2(0, T; H^4)}
		&\leq C \sqrt{\nu} \norm{g_0}_{H^2}.
	\end{split}
    \end{align}
\end{lemma}

\begin{remark}\label{rmkEvenHighRegularity}
Noting that $g_t=\nu\Delta g$ and we have $\sqrt{\nu} \norm{g}_{L^\infty(0,T;H^2)}^2\leq C\sqrt{\nu}\norm{g_0}_{H^2}^2$ we can use the lifting lemma to obtain $\prt_t z\in L^\infty(0,T; H^1)$.
Moreover, we have $g_{tt}=\nu\Delta g_t$ for $t>0$ provided that $g_0\in H^2$ with $\nabla (\Delta g_0)\cdot n=0$. Then, using the lifting lemma and the previous result we obtain
$\prt_{tt} z\in L^2(0,T;L^2)$.
\end{remark}

\end{document}